\documentclass[11pt]{amsart}

\usepackage{amssymb}                        
\usepackage{latexsym}
\usepackage{xspace}                         
\usepackage{stmaryrd}                       
\usepackage{calc}

\usepackage[margin=1.3in]{geometry}

\newtheorem{theorem}{Theorem}[section]

\numberwithin{equation}{section}

\theoremstyle{plain}

\newtheorem*{theorem*}{Theorem}
\newtheorem{proposition}[theorem]{Proposition}

\newtheorem{lemma}[theorem]{Lemma}

\theoremstyle{definition}
\newtheorem{definition}[theorem]{Definition}

\newtheorem{remark}[theorem]{Remark}

\theoremstyle{remark}

\newcommand{\new}{\newcommand}
\newcommand{\ren}{\renewcommand}

\new{\extraskip}{\bigskip} 
\new{\extranewpage}{\newpage}
\new{\extravspace}{\vspace}

\renewcommand{\AA}{\mathbb{A}}

\newcommand{\FF}{\mathbb{F}}

\newcommand{\PP}{\mathbb{P}}

\newcommand{\ZZ}{\mathbb{Z}}

\newcommand  {\shF}     {\mathcal{F}}

\newcommand  {\shL}     {\mathcal{L}}

\newcommand  {\shS}     {\mathcal{S}}

\newcommand  {\fom}     {\mathfrak{m}}

\newcommand  {\fop}     {\mathfrak{p}}

\newcommand  {\Ext}     {\operatorname{Ext}}

\newcommand  {\modu}     {\operatorname{mod}}

\newcommand  {\nor}     {{\operatorname{nor}}}

\renewcommand{\O}       {\mathcal{O}}

\newcommand  {\ord}     {\operatorname{ord}}

\newcommand  {\Proj}    {\operatorname{Proj}}

\newcommand  {\Spec}    {\operatorname{Spec}}

\newcommand  {\Syz}     {\operatorname{Syz}}

\newcommand{\comdots}{ , \ldots , }

\newcommand{\cal}[1]{\mathcal{#1}}

\newcommand{\longto}{\longrightarrow}

\renewcommand{\phi}{\varphi}        
\renewcommand{\epsilon}{\varepsilon}

\newcommand{\tensor}{\otimes}

\newcommand{\set}[1]{\left\{#1\right\}}

\renewcommand{\to}[1][]{\xrightarrow{\ #1\ }}

\usepackage{amscd}
\usepackage{amssymb}

\ren{\sec}[1]{{ {#1}'}}

\newcommand{\elef}{{f}}
\newcommand{\eleg}{{g}}

\newcommand{\elek}{{k}}
\newcommand{\elel}{{l}}
\newcommand{\elem}{{m}}
\newcommand{\elen}{{n}}

\newcommand{\eler}{{r}}

\newcommand{\elet}{{t}}

\newcommand{\eleu}{{u}}

\newcommand{\eletest}{{z}} 

\ren{\elef}{t} 

\ren{\eleg}{s}

\ren{\elel}{t} 

\ren{\elem}{s} 

\ren{\elen}{t}

\ren{\elet}{y}

\newcommand{\expox}{{r}}
\newcommand{\expoy}{{s}}

\newcommand{\ind}{{\kappa}}

\newcommand{\indsec}{{j}}

\newcommand{\indi}{i}

\newcommand{\indj}{j}

\newcommand{\Frob}{\Phi} 

\newcommand{\expoe}{{e}}

\newcommand{\algforc}{B}

\newcommand{\modra}{\mu}
\newcommand{\submodra}{\nu}

\newcommand{\spax}{X}

\newcommand{\alg}{{A}}

\newcommand{\ringr}{R}

\newcommand{\shinv}{\shL}

\newcommand{\covmap}{\varphi}

\newcommand{\covmapcov}{\covmap_i: X_i \to X}

\new{\idbla}{{H}}
\new{\idblasec}{{G}}
\new{\idblatri}{{E}}

\newcommand{\spay}{Y}

\newcommand{\spac}{{C}}

\newcommand{\point}{{P}}

\newcommand{\open}{{U}}

\new{\normcomp}{{|}}

\new{\gen}{{g}}  

\new{\classres}{{u}} 

\new{\classcoho}{{c}}

\new{\homcon}{{\delta}} 

\new{\ringnorm}{{\ring^\nor}}

\new{\spaxnorm}{{\spay}}

\new{\linel}{{L}}

\newcommand{\field}{{K}}

\new{\cloalg}{\overline}

\newcommand{\fieldres}{{\kappa}}

\newcommand{\ring}{{R}}

\newcommand{\ringsec}{{S}}

\newcommand{\fuf}{{f}}

\newcommand{\var}{z}

\newcommand{\varx}{z}
\newcommand{\vary}{w}
\newcommand{\varz}{v}

\newcommand{\varu}{u}
\newcommand{\varv}{v}
\newcommand{\varw}{w}

\newcommand{\degm}{{m}}

\newcommand{\expofu}{{n}}

\new{\expoideal}{{d}}

\new{\varcoef}{{t}}  

\new{\fucoef}{{g}}

\new{\indcoef}{{\nu}}

\new{\vargen}{{z}}

\newcommand{\ideal}{{I}}

\newcommand{\idealsec}{{J}}

\newcommand{\numgen}{{n}} 

\newcommand{\indseq}{{n}}  

\new{\fuarb}{{g}}

\newcommand{\varfu}{q}

\newcommand{\multexp}{{\gamma}}

\newcommand{\monom}{{\var^\multexp}}

\newcommand{\indaxisrun}{{\iota}}

\ren{\set}{{S}}

\ren{\modra}{{m}}

\ren{\submodra}{{\numgen}}

\ren{\expofu}{{k}}

\ren{\indaxisrun}{{i}}

\ren{\varfu}{{w}}

\ren{\indseq}{{m}}

\renewcommand{\field}{{K}}

\renewcommand{\varu}{{U}}
\renewcommand{\varx}{{x}}
\renewcommand{\vary}{{y}}
\renewcommand{\varz}{{z}}

\newcommand{\genus}{{g}}
\newcommand{\sysmult}{{M}}

\renewcommand{\point}{{\fop}}

\ren{\expox}{{i}}
\ren{\expoy}{{j}}

\new{\poly}{{P}}

\new{\polycu}{{\varz^4+\varx \vary \varz^2 + \varx^3 \varz + \vary^3 \varz}}

\new{\polycubra}{{\varz^4+(\varx \vary) \varz^2 + (\varx^3 + \vary^3) \varz}}

\new{\fieldclo}{{L}}

\new{\ringgen}{{R}} 

\new{\vartrans}{{t}}

\new{\equat}{{g}}

\new{\equa}{{\equat}}

\new{\equatrans}{{\equat_\vartrans}}

\new{\varxyz}{{\varx,\vary,\varz}}

\new{\ringpoly}{{\fieldclo[\varxyz]}}
\new{\extratrans}{{\vartrans \varx^2 \vary^2}}

\new{\idealcu}{{(\varx^4,\vary^4,\varz^4)}}

\new{\ele}{{f}}

\new{\elecu}{{\vary^3\varz^3}}

\ren{\sysmult}{{S}}

\new{\sysmultcu}{{\fieldclo [\vartrans]   -\{0\}}}

\new{\fieldquot}{{Q}}

\new{\po}{{Q}}

\new{\pof}{{q}}

\new{\fourpo}{{4 \po}}

\new{\idealgenfourpo}{{\varx^{\fourpo}, \vary^{\fourpo},
\varz^{\fourpo}}}

\new{\expx}{{i}}
\new{\expy}{{j}}
\new{\expz}{{k}}

\ren{\alg}{{\alpha}}
\new{\equalg}{{\equat_\alg}}

\new{\degalg}{{m}}
\new{\ringo}{{\cal O}}
\new{\maxfropof}{{(\varx^\pof,\vary^\pof, \varz^\pof) }}
\new{\elerep}{{v}}
\new{\testcu}{{\varx \vary}}
\new{\varyz}{{\vary \varz}}
\ren{\monom}{{\varx^\expx \vary^\expy \poly^\elek}}
\new{\monomvar}{{\varx^\expx \vary^\expy \poly^\povar}}
\new{\coef}{{c}}
\new{\ringglobal}{{A}}
\new{\polyr}{{R}}
\new{\triple}{{(\expx, \expy, \elek)}}
\new{\polys}{{S}}
\new{\povar}{{q_0}}
\ren{\triple}{{(\expx, \expy, \elek)}}
\new{\triplevar}{{(\expx, \expy, \povar)}}
\new{\subw}{{W}}
\new{\subwsp}{{W_0}}
\new{\subd}{{D}}
\new{\varxex}{{\varx^\expx}}
\new{\varyex}{{\vary^\expy}}
\new{\varzex}{{\varz^\expz}}
\new{\subwpri}{{\subw'}}
\new{\subn}{{N}}
\new{\subx}{{X}}
\new{\suby}{{Y}}
\new{\neqz}{{\neq 0}}
\new{\spacequot}{{\subw/\subwsp}}
\new{\spacequotpri}{{\subwpri/\equa \subwsp}}
\new{\spacequotmap}{{\spacequot \to \spacequotpri}}
\new{\base}{{E_\expox}}
\new{\basec}{{E_\elek}}
\new{\baspri}{{F_\expox}}
\new{\pofsp}{{q_0}}
\new{\spaxtoy}{{\spax\to \spay}}
\new{\varyxz}{{\vary,\varx,\varz}}
\new{\matrixent}{{m_{\indi, \indj}}}
\new{\matrixdet}{{(m_{\indi, \indj})}}
\new{\indij}{{\indi,\indj}}
\new{\coefb}{{b}}
\new{\ordk}{{k}}
\new{\idmaxgenbrafropof}{{(\varx^\pof, \vary^\pof, \varz^\pof)}}
\new{\elealg}{{\alg}}
\new{\elesec}{{\beta}}
\new{\degele}{{m}}
\new{\ringext}{{T}}
\ren{\field}{{F}}
\new{\unitfield}{{\field^\times}}
\new{\numprim}{{p}}
\new{\idealtest}{{\tau}}
\new{\degequa}{{d}}
\new{\degclass}{{k}}
\new{\idealmax}{{\fom}}
\new{\Frobit}{{\Frob^\expoit}}
\new{\Frobitpb}{{\Frob^{\expoit*}}}
\new{\expoit}{{e}}
\new{\ringone}{D}
\new{\ringoneunit}{\ringone^\times}
\new{\idmax}{\fom}
\new{\elnum}{{h}}
\new{\ringtrans}{\ring_\vartrans}
\new{\monomtri}{{\varxex \varyex \varzex}}
\new{\froposec}{{\tilde{q}}}
\new{\expo}{{m}}
\new{\qup}{{q'}}
\new{\pointgen}{{\eta}}
\new{\ringglobalsec}{{B}}
\new{\elesys}{{h}}

\begin{document}

\title[Tight closure does not commute with localization]
{Tight closure does not commute with localization}

\author[Holger Brenner]{Holger Brenner}
\address{Fachbereich Mathematik und Informatik, Universit\"at Osnabr\"uck,
Albrechtstr. 28a, 49069 Osnabr\"uck, Germany}

\author[Paul Monsky]{Paul Monsky}
\address{Mathematics Department, MS 050,
Brandeis University,
Waltham, MA 02254-9110}

\email{hbrenner@uni-osnabrueck.de, monsky@brandeis.edu}

\begin{abstract}
We give an example showing that tight closure does not commute with
localization.
\end{abstract}

\maketitle

\noindent Mathematical Subject Classification (2000): 13A35, 14H60

\section*{Introduction}
{\ren{\po}{{q}}
\ren{\eletest}{{t}}

At the outset of chapter 12 of \cite{hunekeapplication},
Huneke declares:
``This chapter is devoted to the most frustrating
problem in the theory of tight closure. From the first day it was
clearly an important problem to know that tight closure commutes
with localization.''

The reason for Huneke's frustration in establishing the result is now
clear. It is not always true, and our paper provides the first
counterexample.

We recall the notion of tight closure, which was
introduced by M. Hochster and C.
Huneke some twenty years ago and is now an important tool in commutative algebra
(see \cite{hochsterhunekebriancon}, \cite{hunekeapplication}, \cite{hunekeparameter}).
Suppose that $\ring$ is a commutative noetherian domain containing a field
of positive characteristic
$\numprim >0$. Then the \emph{tight
closure} of an ideal $\ideal$ is defined to be
$$\ideal^*= \{\ele \in \ring:
\mbox{ there exists } \eletest
\neq 0 \mbox{ such that } \eletest \ele^\po \in \ideal^{[\po]}
\mbox{ for all } \po=\numprim^\expoe \}\,.$$
Here $\ideal^{[\po]} = (\ele^\po:\ele \in \ideal)$ is the ideal generated
by all $\ele^\po$, $\ele$ in $\ideal$.
The localization problem is the
following: suppose that $\sysmult \subseteq \ring$ is a
multiplicative system and $\ideal$ is an ideal of $\ring$.
Is
$(\sysmult^{-1} \ideal)^* = \sysmult^{-1}(\ideal^*)$?
That $\sysmult^{-1}(\ideal^*)$ is contained in $(\sysmult^{-1} \ideal)^*$
is trivial; the other inclusion is the problem.
The question is: if
$\ele$ is in $( \sysmult^{-1} \ideal)^*$ must there be an $\elesys \in \sysmult$ with
$\elesys \ele $ in $\ideal^*$?

Various positive results for the localization problem are mentioned in chapter 12 of \cite{hunekeapplication} (see also \cite{smithbinomial} and \cite{hochsterhuneketestexponent} for further positive results).
One attack that has had successes is through plus closure.
This approach works when the tight closure of $\ideal$ coincides with its
plus closure - that is to say when for each $\ele \in \ideal^*$
there is a finite domain extension $\ringext$ of $\ring$ for which
$\ele$ is in $\ideal\ringext$.
Tight closure is plus closure for parameter ideals in an excellent
domain \cite{smithparameter} and for graded
$\ring_+$-primary ideals in a two-dimensional standard-graded domain over
a finite field \cite{brennertightplus}.

In this paper we give an example of a three-dimensional
normal hypersurface domain in characteristic two, together with an
ideal $\ideal$, an element $\ele$ and a multiplicative system
$\sysmult$, such that $\ele$ is in $(\sysmult^{-1}\ideal)^*$, but $\ele$ is
not in
$ \sysmult^{-1}(\ideal^*)$.
This implies also that $\ele \not\in (\sysmult^{-1}\ideal)^+$,
so also Hochster's ``tantalizing question''
(see \cite[Remarks after Theorem 3.1]{hochstertightsolid}) whether tight closure is
plus closure has a negative answer.

Our example has no direct bearing on the question whether weakly
$F$-regular rings are $F$-regular. Recall that a noetherian ring of positive
characteristic is \emph{weakly $F$-regular}, if all ideals are
tightly closed ($\ideal= \ideal^*$), and \emph{$F$-regular}, if this
holds for all localizations. These notions have deep connections to
concepts of singularities (like log-terminal, etc.) defined in
characteristic zero in terms of the resolutions of the
singularities; see \cite{hunekeparameter}, \cite[Chapter 4]{hunekeapplication} or \cite{smithvanishingsingularities} for these relations. The equivalence of $F$-regular and weakly $F$-regular is known in the Gorenstein case, in the graded case \cite{lyubezniksmithgraded} and over an uncountable field (Theorem of Murthy, see \cite[Theorem 12.2]{hunekeapplication}).
It is still possible that tight closure always commutes with localization at a single element,
namely that $(\ideal_\fuf)^* = (\ideal^*)_\fuf$, and that for algebras of finite type over a finite
field we always have $\ideal^*= \ideal^+$.

Our argument rests on a close study of the homogeneous co-ordinate
rings of certain smooth plane curves and has the following history.
The first serious doubts that tight closure might not be plus closure
arose in the work of the first author in the case of a
standard-graded domain of dimension two. Both closures coincide
under the condition that the base field is finite, but this
condition seemed essential to the proof; this suggested looking at a
family
$\Spec \ringglobal \to \AA^1_{\FF_\numprim}$
of two-dimensional rings parametrized by the affine line.
The generic fiber ring is then a localization of the three
dimensional ring $\ringglobal$ and the generic fiber is defined over
the field of rational functions $\FF_\numprim (\vartrans)$, whereas
the special fibers are defined over varying finite fields. If
localization held, and if an element belonged to the tight closure
of an ideal in the generic fiber ring, then this would also hold in
almost all special fiber rings (Proposition \ref{deform}).
Since the fiber rings are graded of dimension two, one may
use the geometric interpretation of tight closure in terms of vector
bundles on the corresponding projective curves to study tight
closure in these rings.

Now in \cite{monskyhilbertkunzpointquartic} the second author had used
elementary methods to work out the Hilbert-Kunz theory of
$\ring_\alg = \field[\varx,\vary,\varz]/(\equa_\alg)$, where
${\rm char} \field = 2$, and
$$\equa =\equa_\alg = \poly + \alpha \varx^2\vary^2 \,\,\,\mbox{ with }\,\,\, \poly
=\varz^4+ \varx\vary\varz^2 + \varx^3\varz +
\vary^3\varz$$ and $ \alg$ in $\unitfield$.
In particular \cite{monskyhilbertkunzpointquartic} shows that the
Hilbert-Kunz multiplicity
of $\equa$ is $3$ when $\alg$ is transcendental over $\ZZ/(2)$ and
$>3$ otherwise.

\ren{\po}{{Q}}

Seeing these results, the first author realized that the family
$\Spec \field[\varx,\vary,\varz, \vartrans]/(\equa_\vartrans) \to \Spec \field[\vartrans]$
might be a good place to look for a counterexample to the
localization question. The candidate that arose was the ideal
$$\ideal = (\varx^4,\vary^4,\varz^4)\, \, \, \mbox{ and the element } \, \, \,\fuf =\vary^3
\varz^3 \, .$$
It followed directly from the results of \cite{monskyhilbertkunzpointquartic}
that $\fuf \in \ideal^*$ holds in $\ring_\alg$ when $\alg$ is
transcendental (Theorem \ref{inclusiontrans}).
The first author established, using an ampleness criterion due to Hartshorne and Mumford,
that every element of degree $\geq 2$ was a test element in
$\ring_\alg$
(Theorem \ref{testideal} via Lemma \ref{frobeniusinjective}; we later discovered that the argument could be simplified \-- see Remark \ref{remarkextra}). Computer
experiments showed that
$\varx\vary  (\vary^3 \varz^3)^\po \not\in \ideal^{[\po]}$ in $\ring_\alg$, where
$\alg $ is algebraic over $\ZZ/(2)$ of not too large degree,
and $\po$ is a power of $2$ depending on algebraic properties of $\alg$.

The second author built on \cite{monskyhilbertkunzpointquartic} to establish
that this non-inclusion holds in fact for arbitrary algebraic
elements, completing the proof. The argument is presented in section
\ref{noninclusion}. Section \ref{sectionremark} consists of remarks
and open questions. A more elementary but less revealing variant of our presentation is given in \cite{monskyfailure}.

\ren{\field}{{F}}

We thank A. Kaid (University of Sheffield), M. Kreuzer and D. Heldt (both Universit\"at Dortmund)
for their support in the computations with the computer-algebra
system CoCoA which provided strong numerical evidence in an earlier
stage of this work. We also thank the referees for their useful comments.}

\ren{\eletest}{{z}}

\section{Geometric deformations of tight closure}
\label{preparations}

Before we present our example we describe a special case of the localization problem,
namely the question of how tight closure behaves under geometric
deformations.
Let $\field$ be a field of positive characteristic and
$\field[\vartrans] \subseteq \ringglobal$, where $\ringglobal$ is a
domain of finite type. Suppose that an ideal $\ideal$ and an element
$\ele$ are given in $\ringglobal$. Then for every point $\point \in \Spec \field[\vartrans]=\AA^1_\field$
with residue class field $\fieldres(\point)$ one can consider the
(extended) ideal $\ideal$ and $\ele$ in $\ringglobal
\tensor_{\field[\vartrans]}\fieldres(\point)$ and one can ask
whether
$\ele \in \ideal^*$ in the co-ordinate ring of the fiber over
$\point$.
We call such a situation a
\emph{geometric} or \emph{equicharacteristic deformation} of tight closure
(for arithmetic deformations see Remark \ref{arithremark}).
The following proposition shows that if tight
closure commutes with localization, then also tight closure behaves
uniformly under such geometric deformations.

\begin{proposition}
\label{deform}
Let $\field$ be a field of positive characteristic,
let $\ringone \subseteq \ringglobal$ be domains of finite type over
$\field$ and suppose that $\ringone$ is one-dimensional. Let
$\ideal$ be an ideal and $\ele$ an element in $\ringglobal$.
Suppose that $\ele \in \ideal^* $ in the generic fiber ring
${\ringoneunit}^{-1}\ringglobal$. Assume that tight closure commutes with localization. Then
$\ele \in \ideal^* $ holds also in the fiber rings
$\ringglobal \tensor_\ringone \ringone/\idmax$ for almost all
maximal ideals $\idmax$ of $\ringone$.
\end{proposition}
\begin{proof}
Suppose that $\ele \in \ideal^*$ in ${\ringoneunit}^{-1}\ringglobal$.
If tight
closure localizes, then there exists an element $\elesys \in
\sysmult={\ringoneunit}$ such that $\elesys \ele \in \ideal^*$ holds in
$\ringglobal$. By the persistence of tight closure \cite[Theorem 2.3]{hunekeapplication}
(applied to
$\varphi: \ringglobal \to \ringglobal \tensor_\ringone \ringone/\idmax$)
we have for every maximal ideal $\idmax$ of $\ringone$ that
$\varphi(\elesys)\varphi(\ele) \in (\varphi(\ideal) )^*$
holds in $\ringglobal \tensor_\ringone \ringone/\idmax$.
Since $\elesys$ is contained in only finitely many maximal ideals of $\ringone$, it
follows that $\varphi(\elesys)$ is a unit for almost all maximal
ideals, and so $\varphi(\ele) \in (\varphi(\ideal) )^*$ for almost
all maximal ideals.
\end{proof}

We will apply Proposition \ref{deform} in the situation where $\ringone=\field[\vartrans]$
and $\ringglobal=\field[\vartrans, \varxyz]/(\equa)$, where $\equa$
is homogeneous with respect to $\varxyz$, but depends also on
$\vartrans$. The fiber rings are then two-dimensional homogeneous
algebras $\ring_{\kappa(\point)}$, indexed by $\point \in \AA_\field^1=\Spec
\field[\vartrans]$. If $\field $ is algebraically closed, then these
points correspond to certain values $\alg \in \field$ or to the
generic point $(0)$. In the two-dimensional graded situation we know
much more about tight closure than in general, since we can work
with the theory of vector bundles on the corresponding projective
curve (see \cite{brennertightproj}, \cite{brennerslope},
\cite{brennertightplus}).
In order to establish a counterexample to the localization problem
via
Proposition
\ref{deform} we have to find an ideal
$\ideal$ and an element $\ele$ in $\ringglobal$ such that
$\ele \in \ideal^*$ in $\ring_{\field(\vartrans)}$, but such that
$\ele \not\in \ideal^*$ in $\ring_{\alg}$ for infinitely many algebraic values
$\alg \in \field$. The first part is comparatively easy, the second
part is more difficult and uses also the following result on test
elements (Theorem \ref{testideal}).
Recall (see \cite[Chapter 2]{hunekeapplication}) that an element $\eletest \in \ring$ is a \emph{test element}
for tight closure if for all ideals $\ideal$ and elements $\ele$ we
have that $\ele \in \ideal^*$ if and only if $ \eletest \ele^\pof
\in \ideal^{[\pof]}$ for all $\pof=\numprim^\expoe$.

\renewcommand{\varx}{{x}}
\renewcommand{\vary}{{y}}
\renewcommand{\varz}{{z}}

\begin{lemma}
\label{frobeniusinjective}
Let $\ring = \field [\varx,\vary,\varz]/(\equa)$ be a normal
homogeneous two-dimensional hypersurface ring over an algebraically
closed field
$\field$ of positive characteristic $\numprim$, $\degequa=
\deg(\equa)$.
Suppose that $\numprim > \degequa -3$ and let $\spac = \Proj \ring$
be the corresponding smooth projective curve
of genus $\genus(\spac)$. Then a cohomology class
$\classcoho\neq 0$ in $H^1(\spac , \O_\spac (\degclass))$ for $\degclass <0$
can not be annihilated by any iteration, $\Frobit$, of the absolute
Frobenius morphism $\Phi: \spac \to \spac$.
\end{lemma}
\begin{proof}
Let a cohomology class
$\classcoho \in H^1 (\spac ,\O_\spac (\degclass)) \cong \Ext^1(\O_C, \O_C(\degclass))$
(see \cite[Proposition III.6.3(c)]{haralg})
be given.
This class defines an extension
$0 \to \O_\spac(\degclass) \to \shS \to \O_\spac \to 0$
(see \cite[Exercise III.6.1]{haralg})
with dual extension
$0 \to \O_\spac \to \shF \to \O_\spac (-\degclass) \to 0$.
Assume that $\degclass$ is negative and that $\classcoho$ is not
zero. In such a situation every quotient bundle of $\shF$ has
positive degree: if $\shF \to \shinv$ is a surjection onto a line
bundle $\shinv$ with
$\deg(\shinv) \leq 0$, then either the composed mapping $\O_\spac
\to \shinv$ is the zero map or the identity
(see \cite[Lemma IV.1.2]{haralg}). In the first case we
get an induced map $\O_\spac(-\degclass) \to \shinv$, yielding a
contradiction. In the second case the sequence would split,
contradicting our assumption that $\classcoho \neq 0$.
We also have
$$\deg(\shF)
= -\degclass \deg(\spac)
\geq \degequa > \frac{2}{\numprim} (\frac{(\degequa -2)(\degequa -1 )}{2} -1 )
= \frac{2}{\numprim} (\genus(\spac) -1) \, .$$
Hence by a Theorem of Hartshorne-Mumford
(see \cite[Corollary 7.7]{hartshorneample})
the rank two bundle $\shF$ is ample.
By \cite[Proposition III.1.6]{haramp} every Frobenius pull-back $\Frobitpb(\shF)$ stays ample
and so every quotient sheaf of $\Frobitpb(\shF)$ is ample as well
(see \cite[Proposition III.1.7]{haramp}).
Since $\O_\spac$ is not ample, it follows that the Frobenius
pull-backs of the short exact sequence can not split. This means
that the Frobenius pull-backs of the cohomology class $\classcoho$
are not zero.
\end{proof}

\begin{remark}
\label{remarkextra}
We give a direct proof of Lemma \ref{frobeniusinjective} for our example $\ring=F[x,y,z]/(z^4+xyz^2+x^3z+y^3z+ \alpha x^2y^2)$ ($\alpha \in \field$, $\field$ a field of characteristic two) which avoids the use of vector bundles and of the theorem of Hartshorne-Mumford. We show that for all $k \leq 0$ the Frobenius acts injectively on $H^2_{\idmax}(\ring)$. For $k=0$ a basis for $(H^2_{\idmax}(\ring))_0$ is given by the $\check{\rm C}$ech-cohomology classes
$$\frac{z^2}{xy},\, \frac{z^3}{x^2y} ,\, \frac{z^3}{xy^2} \, .$$
We compute the images under the Frobenius explicitly, yielding
$$\Frob(\frac{z^2}{xy}) = \frac{z^4}{x^2y^2} = \frac{xyz^2+x^3z+y^3z+ \alpha x^2y^2}{x^2y^2} = \frac{z^2}{xy} $$
and
$$\Frob(\frac{z^3}{x^2y}) = \frac{z^6}{x^4y^2} = \frac{z^2(xyz^2+x^3z+y^3z+ \alpha x^2y^2)}{x^4y^2} =
 \frac {z^3}{xy^2} \, . $$
Similarly $\Frob(\frac{z^3}{xy^2})= \frac{z^3}{x^2y} $, and so the Frobenius is a bijection in degree zero. For the negative degrees $k$ we do induction on $-k$. So suppose $c \in H^2_{\idmax}(\ring)_k$ is a cohomology class which is annihilated by the Frobenius $\Frob$. Then also $xc,yc,zc \in H^2_{\idmax}(\ring)_{k+1}$ are annihilated by the Frobenius. So by the induction hypothesis we have $xc=yc=zc=0$. However, the elements in the socle of $H^2_{\idmax}(\ring)$ are exactly the cohomology classes of degree $1$. Therefore $c=0$.
\end{remark}

\begin{lemma}
\label{tightzero}
Let $\ring = \field [\varx,\vary,\varz]/(\equa)$ be a normal
homogeneous two-dimensional hypersurface ring over an algebraically
closed field $\field$ of positive characteristic $\numprim$,
$\degequa= \deg(\equa)$.
Then for $\numprim > \degequa -3$ the tight closure $0^*$ of $0$ in
$H^2_{\idmax} (\ring)$
{\rm(}where $\idmax =(\varxyz)${\rm)} lives only in non-negative
degrees.
\end{lemma}
\begin{proof}
For local cohomology in general we refer to \cite[Section 3.5]{brunsherzog} and for
$0^*$ to \cite[Section 4]{hunekeparameter}.
Let a cohomology class
$\classcoho \in H^2_\idmax(\ring)$ be given.
The local cohomology module $H^2_\idmax(\ring)$ is
$\ZZ$-graded and it is clear that $\classcoho \in 0^*$
if and only if every homogeneous component of the class belongs to $0^*$.
Hence we may assume that $\classcoho$ is homogeneous of degree $\degclass$.
We claim first that for $\classcoho \neq 0$ of degree $\degclass < 0$ no Frobenius
power annihilates $\classcoho$.
Let $\spac = \Proj \ring$ be the corresponding smooth projective curve.
Then we have graded isomorphisms $(H^2_{\idmax}
(\ring))_\degclass \cong (H^1(D(\idmax), \O_{\ring}))_\degclass$
and
$(H^1(D(\idmax), \O_{\ring}))_\degclass \cong H^1(\spac, \O_\spac(\degclass))$, where $D(\idmax)$ is the punctured spectrum $\Spec (\ring) -\{\idmax\}$ (see \cite[Exercises  III.2.3 and III.3.3]{haralg} or \cite[Section 1.3]{smithvanishingsingularities}).
These isomorphisms are compatible with the action of the Frobenius
(see \cite[Section 1.4]{smithvanishingsingularities}).
So the claim follows from Lemma \ref{frobeniusinjective}.

Suppose now that the cohomology class $\classcoho \neq 0$, homogeneous of negative degree $\degclass$,
belongs to $0^*$.
Then also $\Frobitpb(\classcoho) \in 0^*$ for all $\expoit$ and so the test ideal
$\idealtest=\idealtest_\ring$ annihilates
$\Frobitpb(\classcoho)$ for all $\expoit$.
The test ideal $\idealtest$ contains a power of
$\idmax$ by \cite[Theorem 2.1]{hunekeapplication}.
Hence $\ring/\idealtest$ is Artinian and its Matlis dual (see \cite[Section
3.2]{brunsherzog}), which is the submodule of $H^2_\idmax (\ring)$
annihilated by $\idealtest$, is finite by \cite[Theorem
3.2.13]{brunsherzog}
($H^2_\idmax (\ring)$ itself is the injective
envelope of
$\ring/\idmax$ by \cite[Proposition 3.5.4 (c)]{brunsherzog}).
Since the degrees of $\Frobitpb(\classcoho)$ ($\neq 0$) go to $-\infty$, we get a contradiction.
\end{proof}

\begin{remark}
One can also prove Lemma \ref{tightzero} using the geometric
interpretation of tight closure. By the proof of Lemma
\ref{frobeniusinjective}
we know that the dual extension $\shF$ (corresponding to a non-zero
cohomology class of negative degree) is ample.
Hence the open complement
$\PP(\shF) - \PP( \O_C(-\degclass))$ is an affine scheme by \cite[Proposition
II.2.1]{haramp}. This means by
\cite[Proposition 3.9]{brennertightproj} that $\classcoho \notin 0^*$.

The extension used in the proof of Lemma
\ref{frobeniusinjective} can be made more explicit. Suppose, for ease of notation, that
$\varx$ and $\vary$ are parameters and that the homogeneous
cohomology class of degree $\degclass$ is given as a
$\rm\check{C}$ech cohomology class
$\classcoho = \frac{\elnum}{\varx^\expox\vary^\expoy}$ with $\elnum$ homogeneous of degree $\degm$,
$\degclass =\degm- \expox - \expoy$.
Then the extension is
$$0 \longto \O_\spac(\degclass) \cong \Syz(\varx^\expox, \vary^\expoy)(\degm)
\longto \shS \cong \Syz(\varx^\expox, \vary^\expoy, \elnum)(\degm) \longto \O_\spac \longto
0 \, ,$$
where the identification on the left is induced by $1 \mapsto (\vary^\expoy,- \varx^\expox)$
and the last mapping is the projection to the third component.
This can be seen by computing the corresponding cohomology class via
the connecting homomorphism.
\end{remark}

\begin{theorem}
\label{testideal}
Let $\ring = \field [\varx,\vary,\varz]/(\equa)$ be a normal
homogeneous two-dimensional hypersurface ring over an algebraically
closed field $\field$ of positive characteristic $\numprim$,
$\degequa= \deg(\equa)$.
Then for $\numprim > \degequa -3$
every non-zero element of degree $ \geq \degequa -2$ is a test
element for tight closure.
\end{theorem}
\begin{proof}
The test ideal $\idealtest$ is the annihilator of the tight closure $0^*$ inside
$H^2_\fom(\ring )$
by \cite[Proposition 4.1]{hunekeparameter}. By Lemma
\ref{tightzero} we have $0^* \subseteq H^2_\fom(\ring)_{\geq
0}$.
We also have $(H^2_\idmax(\ring))_\degclass = H^1(\spac, \O_\spac(\degclass))= 0$ for $ \degclass \geq
\degequa -2$, since the canonical divisor is $\O_\spac(\degequa-3)$.
Hence $\ring_{\geq \degequa -2}$ multiplies every cohomology class of non-negative
degree into $0$, and therefore $ \ring_{\geq \degequa -2} \subseteq
\idealtest$.
\end{proof}

\begin{remark}
Theorem \ref{testideal} is known to hold for $\numprim \gg 0$ by the
so-called strong vanishing theorem due to Hara (see \cite{hararational} and \cite[Theorem 6.4]{hunekeparameter}). The point here is
the explicit bound for the prime number (Theorem 6.4 in \cite{hunekeparameter} also has an explicit bound obtained by elementary means. Huneke's bound seems to be $p >d-2$, which for our purpose just fails). Note that for $\degequa=4$
every element of degree $\geq 2$ is a test element in arbitrary
characteristic. For $\degequa=5$ it is not true that every element
of degree $\geq 3$ is a test element in characteristic two. For
example, in
$\field[\varxyz]/(\varx^5+\vary^5+\varz^5)$ we have $\varz^3 \in
 (\varx^2,\vary^2)^*$, since
$(\varz^3)^2 = \varz^6 =\varz(\varx^5+\vary^5) \in
(\varx^4,\vary^4)= (\varx^2,\vary^2)^{[2]}$, but
$(\varx\vary\varz ) \varz^3\not\in (\varx^2,\vary^2)$.
\end{remark}

\section{The example}

Throughout, $\fieldclo$ is the algebraic closure of $\ZZ/2$, and we set
$$ \poly
= \varz^4+ \varx\vary\varz^2 + \varx^3\varz +
\vary^3\varz
\,\,\,\mbox{ and }\,\,\,
\equa_\alg = \poly + \alg \varx^2\vary^2
 \, ,$$
where $\alg \in \field$, some field of characteristic two, or
$\alg=\vartrans$, a new variable.

\begin{definition}
\label{exampledef}
Set $\ringglobal = \fieldclo[\varxyz,\vartrans]/(\equatrans)$,
where $\equatrans= \poly +
\extratrans$,
$\ideal= \idealcu$, $\ele = \elecu$ and $\sysmult = \sysmultcu$.
\end{definition}

We will show in this and the next section that these data constitute a
counterexample to the localization property. The following two
results are contained in or are implicit in
\cite{monskyhilbertkunzpointquartic}.

\ren{\ind}{{i}}

\begin{theorem}
\label{theoremoneone}
Let $\ringo$ be the graded ring
$\field[\varxyz]/\idmaxgenbrafropof $ where ${\rm char} (\field)=2$
and $\pof \geq 2$ is a power of $2$.
Suppose that $\alg$ in $\field$ is transcendental over $\ZZ/(2)$.
Then the ring $\ringo/( \equa_\alg )$ is trivial in degree $\geq
3\pof/2 +1$.
\end{theorem}
\begin{proof}
Theorem 4.18 of \cite{monskyhilbertkunzpointquartic}
tells us that $\ringo/(\equa_\alg)$ has dimension
$3 \pof^2-4$, Combining this with Lemma 4.1 of \cite{monskyhilbertkunzpointquartic}
we find that the multiplication by $\equa_\alg$ map,
$\ringo_\ind \to \ringo_{\ind+4}$ is one to one
(injective)
for $\ind \leq 3\pof/2 -4$. Now when $\ind+\indsec=3\pof -3$,
$\ringo_\ind$ and
$\ringo_\indsec$ are dually paired into the one-dimensional vector
space $\ringo_{3\pof-3}$. Furthermore the multiplication by
$\equa_\alg$ maps $\ringo_\ind \to \ringo_{\ind+4}$ and $\ringo_{\indsec-4} \to \ringo_{\indsec}$
are dual. So for $\indsec \geq 3\pof/2 +1$, $\ringo_{\indsec-4} \to
\ringo_{\indsec}$ is onto.
\end{proof}

\begin{theorem}
\label{theoremonetwo}
Suppose $\elealg$ in $\unitfield$ is algebraic over $\ZZ/(2)$. Choose
$\elesec \in \overline{\field}$ so that $\elealg= \elesec^2+\elesec$.
Let $\degele =
\degele (\elealg)$ be the degree of $\elesec$ over $\ZZ/(2)$
{\rm(}Since for each $\elek$,
$\sum_{\ell < \elek} \elealg^{2^\ell} = \elesec ^{2^\elek}
+\elesec$, $\degele$ may be described as the smallest $\elek$ such
that $\sum_{\ell \leq \elek -1}  \elealg^{2^\ell}$ is $0${\rm)}. Let
$\po=2^{\degele -1}$, and
$\ringo= \field[\varxyz]/(\varx^{4 \po},\vary^{4 \po}, \varz^{4 \po})$.
Then in degree $6\po+2$, $\ringo/(\equa_\elealg)$ has dimension $1$.
\end{theorem}
\begin{proof}
Lemma 4.15 of \cite{monskyhilbertkunzpointquartic} with $\pof=4
\po=2^{\degalg+1}$ tells us that the multiplication by $\equa_\alg$
map
$\ringo_{6\po-5} \to \ringo_{6\po-1}$ has one-dimensional kernel. The
duality argument given above shows that
$\ringo_{6\po-2} \to \ringo_{6\po+2}$ has one-dimensional cokernel.
\end{proof}

\ren{\fieldquot}{{F}}

\begin{theorem}
\label{inclusiontrans}
$\ele$ is in $(\sysmult^{-1} \ideal)^*$.
\end{theorem}
\begin{proof}
Let $\fieldquot = \fieldclo(\vartrans)$. Then $\sysmult^{-1}\ringglobal$
identifies with
$\ring_\vartrans = \fieldquot[\varxyz]/(\equatrans)$.
We give two proofs of Theorem \ref{inclusiontrans}, one more conceptual, one more
elementary, but both rest in the end on
\cite{monskyhilbertkunzpointquartic}. First proof: Since by
\cite{monskyhilbertkunzpointquartic} the Hilbert-Kunz multiplicity
of $\ringtrans$ is $3$, the syzygy bundle
$\Syz(\varxyz)$ on $\Proj \ringtrans$ is strongly semistable by
\cite[Corollary 4.6]{brennerhilbertkunz}. Then also its second Frobenius pull-back
$\Syz(\varx^4,\vary^4,\varz^4)$ is strongly semistable. By the
degree bound \cite[Theorem 6.4]{brennerslope} everything of degree
$(4+4+4)/2 =6$ belongs to the tight closure of the ideal $\ideal$ in
$\ringtrans$.

For the second proof, let $\po$
be a power of $2$. Since $\vary \ele^\po$ has degree $6\po +1$,
Theorem \ref{theoremoneone} with $\pof =4 \po$
shows that $\vary \ele^\po$ represents $0$ in
$\fieldquot[\varxyz]/( \idealgenfourpo , \equatrans)$.
In other words, in the ring $ \fieldquot[\varxyz]/(\equatrans)$, the
element
$ \vary \ele^\po $ lies in $\idealcu^{[\po]}$ for all $\po$. This
gives the theorem.
\end{proof}

\begin{theorem}
\label{complete}
In the situation of definition \ref{exampledef}, we have $\ele \in
(\sysmult^{-1} \ideal)^*$ in $\sysmult^{-1} \ringglobal$, but for each $\alg
\in \fieldclo^\times$,
$\ele \not\in
\ideal^*$ in $\ring_\alg= \ringglobal \tensor_{\fieldclo[\vartrans]} \fieldclo$
{\rm(}where $\fieldclo[\vartrans]$ acts on $\fieldclo$ by $\vartrans \mapsto \alg${\rm)}.
Hence tight closure does not commute with
localization.
\end{theorem}
\begin{proof}
The first result is Theorem \ref{inclusiontrans}. Let
$ \alg\in \fieldclo$ be algebraic, $\alg \neq 0$.
Write $\alg=\elesec^2+\elesec$ and set $\degalg = \deg (\elesec)$
(as in Theorem \ref{theoremonetwo}) and $\po=2^{\degalg-1}$.
We will show in the following section that $\varx\vary (\ele^\po)
\not\in \ideal^{[\po]}$. Since by Theorem \ref{testideal} every
element of degree two is a test element, it follows that $\ele
\not\in \ideal^*$ in $\ring_\alg$. By Proposition \ref{deform} this
implies that tight closure does not commute with localization.
\end{proof}

\ren{\elet}{{w}}

\section{A non-inclusion result}
\label{noninclusion}

Let $\alg$ denote a fixed non-zero element of $\fieldclo$, let
$\degalg$ and $\po=2^{\degalg-1}$ be as in Theorem \ref{theoremonetwo},
and set $\pof = \fourpo=  2^{\degalg +1}$. Let
$\ringo$ be the graded ring $\ringpoly/\maxfropof $, and let $\elerep$
be the element of $\ringo_{6 \pof +2}$ represented by
$(\testcu)(\varyz)^{3 \po}$. Our goal is to show that $\elerep
\not\in \equalg \ringo$, thereby proving Theorem \ref{complete}. Our arguments are close to those of
\cite{monskyhilbertkunzpointquartic}, and we assume familiarity with
that paper. We shall write $\equat$ for $\equalg$.

There is a map
$\ringo_{6 \po -5} \to \ringo_{6 \po -1}\oplus \ringo_{12\po -3}$ given by
$\eleu \mapsto (\eleu \equat, \eleu \elerep )$.
If we could show that this map is one to one we would be done. For
the proof of Theorem \ref{theoremonetwo} shows there is a $\elet\neq
0$ in $\ringo_{6 \po -5}$ with $\elet \equa = 0$. Then the one to
oneness shows that $\elet \elerep \neq 0$, and so
$\elerep \not\in \equa \ringo$.
Unfortunately, $\ringo_{6 \po -5} $ and $ \ringo_{6 \po
-1}$ are too large to allow an understanding of the map given above,
and as in
\cite{monskyhilbertkunzpointquartic}
one must proceed by replacing them by subquotients.
Before doing this we calculate some products in $\ringo$.

{\ren{\eler}{{r}}
\begin{lemma}
\label{twopowers}
Let $a_\expx$ be distinct powers of $2$ and $b_\expx$ be powers of
$2$. Suppose $\sum a_\expx b_\expx = 2^\eler-1$. Then each $b_\expx$ is
$1$.
\end{lemma}
\begin{proof}
We argue by induction on $\eler$, $\eler =1$ being trivial.
Since $2^\eler -1$ is odd and at most one $a_\expx$ is odd, we may
assume $a_1=b_1=1$. Then $\sum_{\expx >1} (\frac{a_\expx}{2})
b_\expx =2^{\eler-1}-1$, and induction gives the result.
\end{proof}
}

\begin{lemma}
\label{zexpand}
Suppose that $\elek < \po -1$. Then no monomial appearing with non-zero coefficient in the expansion of
$\poly^\elek$ can have $\po -1$ as the exponent of $\varz$.
\end{lemma}
\begin{proof}
Write $\elek = \sum a_\expx$, where the $a_\expx$
are distinct powers of $2$. Then
$$\poly^\elek= \prod (\polycubra)^{a_\expx} \, .$$
Expanding this product in powers of $\varz$ we see
that all the exponents of $\varz$ that appear are of the form $\sum
a_\expx b_\expx$ with each $b_\expx=1,2$ or $4$.
If $\sum a_\expx b_\expx = \po -1$, Lemma \ref{twopowers}
shows that $\elek= \sum a _\expx = \po -1$, contradicting our
hypothesis.
\end{proof}

\begin{lemma}
\label{monomexpand}
If $\elek < \po -1$ and $ \expx + \expy +4 \elek =6 \po -5$, then
$\elerep \cdot (\varx^\expx \vary^\expy \poly^\elek)=0$ in $\ringo$.
\end{lemma}
\begin{proof}
Since $\elerep= \varx \vary^{3 \po +1} \varz^{3 \po}$
and $\elerep \cdot (\monom )$ lies in
$\ringo_{12\po -3}$, it is $\coef (\varx\vary\varz)^{\pof-1}$ where $\coef$
is the coefficient of $\varx^{\pof-2} \vary^{\po -2} \varz^{\po -1}$
in $ (\monom)$. As $ \elek < \po -1$, Lemma \ref{zexpand} shows that
$\coef =0$.
\end{proof}

We next recall some definitions from
\cite{monskyhilbertkunzpointquartic}.

\begin{definition} (Definition 1.5 in \cite{monskyhilbertkunzpointquartic})
$\polyr =\sum \monomtri$, the sum extending over all triples $\triple$
with $\expz =1$ or $2$, with $\expx+ \expy+ \expz=\pof$ and
$\expx \equiv \expy $ mod $3$. $\polys = \sum \monomvar$, the sum extending
over all triples $\triplevar$ with $\povar$ dividing $\pof/8$,
$\expx+\expy+4\povar = \pof$ and $\expx \equiv \expy $ mod $12 \povar$.
\end{definition}

\begin{theorem}
In $\ringo$, $\poly^\po=\polyr+\polys$ and $\varx^3\polyr = \vary^3
\polyr$.
\end{theorem}
\begin{proof}
This is Theorem 1.9 of \cite{monskyhilbertkunzpointquartic}.
\end{proof}

\begin{definition} (Definitions 3.2 and 3.3 in \cite{monskyhilbertkunzpointquartic})
\begin{enumerate}
\item
$\subw \subseteq \ringo_{6\po-5}$ is spanned by the
$\monom$ with $\expx+ \expy+4\elek=6 \po -5$ and $\elek \leq \po -1$.

\item
$\subwsp \subseteq \subw $ is spanned by $\monom$ with
$\expx+ \expy+4\elek=6 \po -5$ and $\elek < \po -1$.
\end{enumerate}
\end{definition}

Note that $\elerep \cdot \subwsp =0$ by Lemma \ref{monomexpand}.
So $\eleu \mapsto \eleu\elerep$ is a well-defined map $\subw/\subwsp
\to \ringo_{12\po -3}$.

\begin{definition} (Definition 3.4 a) in \cite{monskyhilbertkunzpointquartic})
\label{subddef}
$\subd$ is the subspace of $\ringo_{6\po-1}$
spanned by the $\varx^\expx \vary^\expy \polyr$ with $\expx+\expy =
2 \po -1$.
\end{definition}

Now $\varx^3 \polyr = \vary^3 \polyr $ in $\ringo$.
The numbers $\po-2$ and $2\po-2$ ($\po \geq 2$) represent $0$ and
$2$ (or $2$ and $0$) mod $3$; hence each $\varxex\varyex\polyr$ as
in Definition \ref{subddef} is either $\varx^{\po-2}
\vary^{\po+1}\ringr$, $\varx^{2\po-2} \vary \ringr$ or $\varx
\vary^{2 \po -2} \ringr $.

\begin{definition}
\label{subwprime}
$\subwpri \subseteq \ringo_{6\po-1}$ is spanned by $\subd$ together
with the $\monom$ with $\expox+\expoy+4\elek = 6 \po -1$ and $\elek
\leq \po -1$.
\end{definition}

\begin{lemma}
\label{inclusion}
$\equa \cdot \subw \subseteq \subwpri$.
\end{lemma}
\begin{proof}
This is Theorem 3.5 of  \cite{monskyhilbertkunzpointquartic}.
\end{proof}

\begin{lemma}
\label{kernel}
The map $\subw/\subwsp \to \subwpri/\equa \subwsp$ induced by
multiplication by $\equa$ has one dimensional kernel.
\end{lemma}
\begin{proof}
Theorems 3.10 and 3.1(2) of \cite{monskyhilbertkunzpointquartic}
show that the kernel is non-trivial.
Let $\subn_{6\po -5}$ denote
the kernel of the multiplication by $\equa$ map
$\ringo_{6\po -5} \to  \ringo_{6\po -1}$. We have mentioned in the proof of Theorem \ref{theoremonetwo}
that
$\subn_{6\po -5}$ has dimension $1$. Theorem 3.10 of \cite{monskyhilbertkunzpointquartic}
shows that the kernel of
$\subw/\subwsp \to \subwpri/\equa \subwsp$
identifies with a subspace of $\subn_{6\po -5}$, giving the result.
\end{proof}

\ren{\varu}{{u}}

We now have a more promising approach to showing that $\elerep
\notin \equa \ringo$. Consider the map
$$\subw/\subwsp \to \subwpri/\equa \subwsp \oplus \ringo_{12\po -
3}$$
induced by $\eleu \mapsto (\eleu \equa, \eleu \elerep)$. It
suffices to show that this map is one to one.
For suppose that $\elerep \in \equa \ringo$.
If we take $\elet
\neqz$ in the kernel of $\subw/\subwsp \to \subwpri/\equa \subwsp$,
then
$\elet \mapsto (0,0)$, a contradiction.
It is indeed practical to write down the matrix of the map
$\subw/\subwsp \to \subwpri/\equa \subwsp \oplus \ringo_{12\po-3}$,
but it is a bit more convenient to replace $\subw/\subwsp$ and
$\subwpri/\equa \subwsp$ by subspaces $\subx$ and $\suby$ of
dimensions $\po$ and $\po +1$ respectively.
We first describe bases of $\subw/\subwsp$ and
$\subwpri/\equa \subwsp$.

\begin{lemma}
\label{lemmabasis}
Let $\base =\varxex \varyex \poly^{\po-1}$
where $0 \leq \expox \leq 2 \po -1$ and $\expox +\expoy = 2 \po -1$. Then the $\base$
represent a basis of $\spacequot$.
\end{lemma}
\begin{proof}
This is Theorem 3.11 of
\cite{monskyhilbertkunzpointquartic}.
\end{proof}

\begin{lemma}
\label{lemmabasispri}
Let $\baspri =\varx^{2\po+\expox} \vary^{2\po +\expoy}$
where $0 \leq \expox \leq 2 \po -1$ and $\expox +\expoy = 2 \po -1$. Then the
$\baspri$, together with $\varx^{2 \po-2}\vary^{}\polyr $, $\varx^{}\vary^{2\po-2}\polyr
$ and $\varx^{\po-2}\vary^{\po+1}\polyr $
represent a basis of $\spacequotpri$.
\end{lemma}
\begin{proof}
Theorem 2.8 and 3.14 of \cite{monskyhilbertkunzpointquartic}
tell us that the map $\spacequot \to \subwpri/(\equa \subwsp+\subd)$
induced by multiplication by $\equa$ has  $4$ dimensional kernel.
Combining this with Lemma \ref{kernel} above
we see that the image of $\subd$ in $\spacequotpri$ has dimension
$3$, and that one gets a basis of $\spacequotpri$ by taking elements
of $\subwpri$ representing a basis of $\subwpri/(\equa
\subwsp+\subd)$, and adding $3$ elements that span $\subd$.
But Theorem 3.11 of \cite{monskyhilbertkunzpointquartic}
shows that the $\baspri$ represent a basis of
$\subwpri/(\equa \subwsp+\subd)$.
\end{proof}

\begin{lemma}
\label{imagemult}
The image of $E_\elek$ under the multiplication by $\equa$ map
$\spacequot \to \spacequotpri$ is
$$\alg^\po F_\elek + \sum \alg^\pofsp F_\ell +(\mbox{the element }
\varx^\elek\vary^{2\po -1-\elek} \polyr
\mbox{ of } \subd) \, ,$$ where the sum extends over all pairs $( \pofsp, \ell)$
with $\pofsp $ dividing $\po/2$ and $\ell \equiv \elek $ mod $6
\pofsp$.
\end{lemma}
\begin{proof}
Since $\poly^\po=\polyr+\polys$ in $\ringo$ and $\equa=\poly+ \alg
\varx^2\vary^2$, the image of $E_\elek$ is
$$\varx^\elek \vary^{2\po - \elek-1}( \alg \varx^2\vary^2)\poly^{\po
-1} + \varx^\elek \vary^{2\po - \elek -1} \polys + \varx^\elek
\vary^{2 \po - \elek -1} \polyr \, .$$
But the proof of Theorem 3.13 of \cite{monskyhilbertkunzpointquartic}
shows that the first two of these three terms are $\alg^\po F_\elek$
and $\sum \alg^{\pofsp} F_\ell$.
\end{proof}

Lemmas \ref{lemmabasis}, \ref{lemmabasispri} and \ref{imagemult}
allow one to write down the matrix of the map
$ \spacequot \to \spacequotpri$
explicitly. One also wants to know the matrix of the map
$\spacequot \to \ringo_{12 \po -3}$, $\eleu \mapsto \eleu \elerep$.
This can be read off from:

\begin{lemma}
\label{lemmacongruence}
If $\expox + \expoy =2 \po -1$ then $ (\varxex \varyex \poly^{\po -1}) \cdot \elerep $ is
$\coef (\varx\vary\varz)^{\pof -1}$ where $ \coef =1$ if $\expox
\geq \po +1$ and $\expox \equiv \po +1$  mod $3$, and $\coef =0$ otherwise.
\end{lemma}
\begin{proof}
As in the proof of Lemma \ref{monomexpand} we see that $\coef$ is
the coefficient of
$\varx^{\pof-2} \vary^{\po -2} \varz^{\po -1}$ in
$\varxex\varyex\poly^{\po-1}$. Since $\poly =(\varx^3+\vary^3)\varz
+ (\mbox{higher order terms in } \varz)$,
$\coef$ is the coefficient of $\varx^{\pof-2}\vary^{\po-2}$ in $\varxex\varyex (\varx^3 + \vary^3)^{\po -1}$.
Now $\varxex\varyex (\varx^3 + \vary^3)^{\po -1}$ is the sum of
those monomials in $\varx$ and $\vary$ of total degree $5\po -4$,
such that the $\varx$-exponent is congruent to $\expox$ mod
$3$ and lies between $\expox $ and $\expox+3\po -3$. So if $\expox
\not\equiv \po +1$ mod $3$, $\coef =0$, while if $\expox \equiv \po
+1$ mod $3$,
$\coef =1$ if $\expox \leq 4 \po -2 \leq \expox +3 \po -3$ and is $0$ otherwise. This gives the lemma.
\end{proof}

We now define subspaces $\spax \subseteq \spacequot$ and $\spay \subseteq
\spacequotpri$.

\begin{definition}
$\spax    $ is spanned by $G_i =E_{2i} +E_{2\po -1-2 i} $, $0 \leq i \leq \po -1$.
$\spay    $ is spanned by $H_i =F_{2i} +F_{2\po -1-2 i} $, $0 \leq i \leq \po -1$, together with
$\gamma=( \varx^{2\po -2} \vary +\varx\vary^{2\po-2})\polyr$.
\end{definition}

Lemmas \ref{lemmabasis} and \ref{lemmabasispri} tell us that $\spax$
and $\spay$ have dimensions $\po$ and $\po+1$ respectively.

\begin{lemma}
\label{lemmalong}
\begin{enumerate}
    \item $\spacequotmap$  maps $\spax$ into $\spay$ (the map is induced by multiplication by $g$).
     \item The kernel of $\spacequotmap$ is contained in $\spax$\!.
     Consequently the map $\spaxtoy$ of {\rm(}1{\rm)} has one-dimensional
     kernel.
     \item If the map $\spaxtoy \oplus \ringo_{12\po -3}$,
     $\eleu \mapsto (\eleu \equa ,\eleu \elerep )$ is one to one, then $ \elerep \notin \equa \ringo $.
   \end{enumerate}
\end{lemma}
\begin{proof}
By Lemma \ref{imagemult}, the image of $ E_{2k}$ in $W'$ is $\sum \alpha^{q_0} F_{2\ell}
+x^{2k} y^{2Q-1-2k} R$ where the sum extends over all pairs $(q_0,\ell)$ with $q_0$ dividing $Q$ and $2\ell \equiv 2k \mod 6q_0$. Furthermore the image of $E_{2Q-1-2k}$ is $\sum \alpha^{q_0} F_{2Q-1-2\ell} +x^{2Q-1-2k} y^{2k} R$ where the sum ranges over the same index set. We conclude that the image of $G_k$ in
$\subwpri$ is
$$\sum \alg^{\pofsp}H_\ell + (\varx^{2\elek}
 \vary^{2\po-1-2 \elek} +\varx^{2\po -1-2\elek} \vary^{2 \elek})\polyr \, ,$$
where the sum extends over all pairs $(\pofsp, \ell)$ with $\pofsp $ dividing $\po$ and $\ell \equiv
\elek$ mod $3 \pofsp$.
If $2 \elek \equiv 2 \po-1-2\elek$ mod $3$ then
$\varx^{2\elek} \vary^{2 \po -1-2\elek}\polyr = \varx^{2\po -1-2\elek} \vary^{2\elek}\polyr$
and the second term in the image of $G_\elek$ is $0$.
If $2 \elek \not\equiv 2 \po-1-2\elek$ mod $3$, then one of
$ 2 \elek -(2\po-1-2\elek)$ and $(2\po-1-2\elek)-2 \elek$
is congruent to $1$ mod $3$, and the other to $2$ mod $3$, so the
second term in the image of $G_\elek$ is $\gamma$. We conclude that
the image of $G_\elek$ is $\sum \alg^\pofsp H_\ell$ if
$\elek \equiv 2\po -1$ mod $3$, and $\sum \alg^\pofsp H_\ell + \gamma$
otherwise.

The automorphism $(\varxyz) \mapsto (\varyxz)$ of
$\fieldclo[\varxyz]$ is easily seen to stabilize the kernel of
$\spacequotmap$. Since the kernel has dimension $1$ and $1=-1$ in
$\fieldclo$, the
automorphism acts trivially on the kernel. As it interchanges
$E_\indi$ and $E_{2\po -1-\indi}$, the set of elements of
$\spacequot$ fixed by this automorphism is $\spax$, giving (2).
Finally (3) follows from (2) in the usual way.
\end{proof}

\begin{lemma}
\label{lemmanext}
Suppose $0 \leq \elek \leq \po-1$. Then $G_\elek \cdot \elerep
=(\varx\vary\varz)^{\pof-1}$ if $\elek \equiv 2 \po -1$ mod $3$, and
is $0$ otherwise.
\end{lemma}
\begin{proof}
If $\elek \equiv 2\po -1$ mod $3$, both $2 \elek$ and $2
\po-1-2\elek$ are $\equiv \po +1$ mod $3$. Furthermore just one of
$2\elek$ and $2\po -1-2\elek$ is $> \po$, and we apply Lemma
\ref{lemmacongruence}.
If $\elek\not\equiv 2\po -1$ mod $3$, neither $2\elek$ nor $2 \po
-1-2\elek$ is congruent to $\po+1$ mod $3$, and we again use Lemma
\ref{lemmacongruence}.
\end{proof}

Combining the formulas derived in Lemmas \ref{lemmalong} and \ref{lemmanext} we
get:

\begin{theorem}
\label{imageg}
The image of $G_\elek$ under the map $\spax \to \spay \oplus \ringo_{12 \po-3}$,
$\eleu \mapsto (\eleu \equa ,
\eleu \elerep)$ is
$$ \sum \alg^\pofsp H_\ell + b_\elek (\gamma +(\varx\vary\varz)^{\pof-1}) + \gamma \, ,$$
where $b_\elek =1$ if $\elek \equiv 2\po-1$ mod $3$ and is $0$ otherwise.
The sum extends over all pairs $(\pofsp,\ell)$
with $\pofsp$ dividing $\po$ and $\ell \equiv \elek $ mod $3\pofsp$.
\end{theorem}

\ren{\matrix}{{M[ \po, \alg]}}

\begin{definition}
\begin{enumerate}
\item
$\matrix$ is the matrix $\matrixdet$, $0 \leq \indij < \po$,
where $\matrixent= \sum \alg^\pofsp$, the sum extending over all
$\pofsp$ such that $\pofsp$ divides $\po$ and $3 \pofsp$ divides $\indi-\indj$.

\item
$C(\po)$ and $B(\po)$ are the row vectors
  $(\coef_0 \comdots \coef_{\po-1})$ and
  $(\coefb_0 \comdots \coefb_{\po-1})$,
  where each $\coef_\indi$ is $1$, and
  $\coefb_\indi$ is $1$ if $\indi \equiv 2 \po -1$ mod $3$, and is
  $0$ otherwise.
\end{enumerate}
\end{definition}

The matrix $\matrix$ appeared in section 2 of
\cite{monskyhilbertkunzpointquartic}; it will also play a key role
here. We shall use:

\begin{theorem}
\label{entries}
The $\matrixent$ above satisfy:

\begin{enumerate}
  \item $\matrixent =0$ if $\indi= \indj$ or $\indi \not\equiv \indj
  $ mod $3$.

  \item
  Suppose $\indi \equiv \indj
  $ mod $3$ and $\indi \neq \indj$. Then $\matrixent \neq 0$.
  Furthermore $\matrixent$ only depends on the integer $\ord_2( \indi
  -\indj)$.
\end{enumerate}
\end{theorem}
\begin{proof}
When $\indi \not\equiv \indj
$ mod $3$ the sum is empty and $\matrixent =0$. When
$\indi \equiv \indj
$ mod $3$ and $\indi \neq \indj$ let $\ordk =\ord_2( \indi-\indj)$. It's easy to see that $\ordk < \degalg -1$. Now the sum
defining $\matrixent$ is $\sum_{\ell=0}^\ordk \alg^{2^\ell}$, and
the description of $m(\alg)$ given in Theorem \ref{theoremonetwo}
shows that this is $\neq 0$; it only depends on $\ordk$. Finally
when $\indi=\indj$,
$\matrixent = \sum_{\ell =0}^{m-1}
\alg^{2^\ell}$ and again the description of $m(\alg)$ shows that this is $0$.
\end{proof}

\begin{theorem}
\label{matrixdescript}
Let $M=\matrix$. Then with respect to appropriate bases, the matrix
of the map
$\spax \to \spay \oplus \ringo_{12\po -3}$ is
$\left(
   \begin{array}{c}
     M \\
     B(\po) \\
     C(\po) \\
   \end{array}
 \right)
 $.
\end{theorem}
\begin{proof}
The $G_0 \comdots G_{\po-1}$ form a basis of $\spax$ while $H_0 \comdots
H_{\po -1}, \gamma +(\varx\vary\varz)^{\pof-1}$ and $\gamma$ form a
basis of $\spay \oplus \ringo_{12\po -3}$. Now use Theorem
\ref{imageg}.
\end{proof}

\ren{\field}{{F}}

For
the rest of this section $\field$ is a field of characteristic two,
and $\po \geq 2$ is a power of $2$.

\begin{definition}
\label{specialmatrix}
A matrix $M=\matrixdet$, $0 \leq \indij <\po$, with entries in
$\field$, is a \emph{special $\po$-matrix}
if it satisfies (1) and (2) of Theorem \ref{entries}.
\end{definition}

\begin{lemma}
\label{twentythree}
Let $M$ be a special $\po$-matrix with $\po \geq 4$. Write $M$ as
$
\left(
  \begin{array}{ccc}
    M_1 & M_2 & M_3 \\
    M_4 & N & M_5 \\
    M_6 & M_7 & M_8 \\
  \end{array}
\right)$
where $M_1$ and $M_8$ are $\po/4 \times \po/4$ matrices. Then:

\begin{enumerate}
  \item
  $N$ is a special $\po/2$-matrix.
  \item $M_1 =M_8$, $M_2 =M_7$, $M_3 =M_6$ and $M_4 =M_5$.
  \item
  $M_1 + M_3$ is a non-zero scalar matrix.
\end{enumerate}
{\rm(}We are abusing language somewhat. By {\rm(1)} we mean that the matrix
$N=(n_{\indi,\indj})$, $0 \leq \indij < \po/2$, with
$n_{\indi,\indj}= m_{\indi+ \po/4, \indj + \po/4}$ is a special
$\po/2$-matrix. Similarly {\rm(2)} should be interpreted as stating the
equality of certain entries of $M$.{\rm)}
\end{lemma}
\begin{proof}
Let $Q=4q$, and suppose $ 0 \leq i,j<q$. When $i  \neq j \mod 3$, $m_{i,j} =m_{i,j+3q} =0$ and so
the $(i,j)$ entry of $M_1+M_3$ is $0$. When $i \equiv j \mod 3$ but  $i \neq j$, $\operatorname{ord}_2 (q)> \operatorname{ord}_2 (i-j)$, $ \operatorname{ord}_2 (i-j) = \operatorname{ord}_2 (i-j-3q)$ and Definition \ref{specialmatrix} tells us that the $(i,j)$ entry, $m_{i,j}+ m_{i,j+3q}$, of $M_1+M_3$ is $0$. Finally, a diagonal element of $ M_1+M_3$ has the form $m_{i,i} +
m_{i,i+3q} =  m_{i,i+3q}$, which is independent of $ i$. This gives (3). The proofs of (1)
and (2) are equally easy.
\end{proof}

\begin{theorem}
\label{nullity}
A special $\po$-matrix has nullity two.
\end{theorem}
\begin{proof}
In a slightly different notation this is the second half of Theorem
2.4 of
\cite{monskyhilbertkunzpointquartic}.
We repeat the inductive proof given there. When $\po=2$,
$M= \left(  \begin{array}{cc}                       0 & 0 \\       0 & 0 \\  \end{array}   \right)
$.
When $\po \geq 4$, Lemma \ref{twentythree}
allows us to write
$$M= \left(
       \begin{array}{ccc}
         M_1 & D & M_3 \\
         E & N & E \\
         M_3 & D & M_1 \\
       \end{array}
     \right) \,.
$$
Making elementary row and column operations we get the matrix
$$ \left(
       \begin{array}{ccc}
         M_1 & D & M_1+M_3 \\
         E & N & 0 \\
         M_1 + M_3 & 0 & 0 \\
       \end{array}
     \right) \, .
$$
Since $M_1 +M_3$ is a non-zero scalar matrix, further elementary
operations yield
$$ \left(
       \begin{array}{ccc}
         0 & 0 & M_1+M_3 \\
         0 & N & 0 \\
         M_1 + M_3 & 0 & 0 \\
       \end{array}
     \right) \, .
$$
So $M$ and $N$ have the same nullity, and we use induction.
\end{proof}

\begin{theorem}
\label{final}
If $M$ is a special $\po$-matrix, then $\left(
   \begin{array}{c}
     M \\
     B(\po) \\
     C(\po) \\
   \end{array}
 \right)
$
has rank $\po$.
This implies that $\spax \to \spay \oplus\ringo_{12\po-3}$ is one to one,
that $\elerep \notin \equa \ringo$, and that
$(\varx\vary) f^\po$ is not in $(\varx^{4\po},\vary^{4\po},\varz^{4\po}, \equalg)$.
\end{theorem}
\begin{proof}
This is a variation on the proof of Theorem \ref{nullity}.
Since $B(2)$ and $C(2)$ are $(1,0)$ and $(1,1)$ the result holds for
$\po=2$. Now suppose $\po \geq 4$ and write $M$ as
$$\left(
       \begin{array}{ccc}
         M_1 & D & M_3 \\
        E & N & E \\
         M_3 & D & M_1 \\
       \end{array}
     \right) \,.
$$
Set $C(1)=(1)$, $B(1)=(0)$. Evidently $C(\po)=(C(\po/4)|  C(\po/2) | C(\po/4)   )$
and it's easy to see that
$B(\po)=(B(\po/4)|  B(\po/2) | B(\po/4))$.
For example, $\coefb_{\indi+\po/4} =1$ if and only if
$\indi +\po/4 \equiv 2\po -1 $ mod $3$, i.e. if and only if
$\indi \equiv 2(\po/2)-1$ mod $3$. Making the same elementary row
and column operations on
$\left(
   \begin{array}{c}
     M \\
     B(\po) \\
     C(\po) \\
   \end{array}
 \right)
 $ as we made on $M$ in the proof of Theorem \ref{nullity},
 we get the matrix
 $$ \left(
      \begin{array}{ccc}
        M_1 & D & M_1 +M_3 \\
        E & N & 0 \\
        M_1 +M_3 & 0 & 0 \\
        B(\po/4) & B(\po/2) & 0 \\
        C(\po/4) & C(\po/2) & 0 \\
      \end{array}
    \right)
\, . $$
Using the fact that $M_1+M_3$ is a non-zero scalar matrix we make
further elementary operations yielding

 $$ \left(
      \begin{array}{ccc}
        0 & 0 & M_1 +M_3 \\
        0 & N & 0 \\
        M_1 +M_3 & 0 & 0 \\
        0 & B(\po/2) & 0 \\
        0 & C(\po/2) & 0 \\
      \end{array}
    \right)
\, . $$
The rank of this matrix is $2 \cdot(\po/4) + {\rm rank} \left(
   \begin{array}{c}
     N \\
     B(\po/2) \\
     C(\po/2) \\
   \end{array}
 \right)$
and induction completes the proof.
The other conclusions follow from Theorem \ref{matrixdescript},
Lemma \ref{lemmalong}(3), and the definition of $\elerep$.
\end{proof}

\section{Some consequences and remarks}
\label{sectionremark}

\begin{remark}
In
$\ringglobal =
\fieldclo[\varxyz, \vartrans]/(\equa_\vartrans)$, an $\sysmult$-multiple of
$\ele= (\vary \varz)^3$ is not in the
tight closure of
$(\varx^4,\vary^4,\varz^4)$, and so is not in the plus closure of
$(\varx^4,\vary^4,\varz^4)$ either. Since plus closure commutes with localization,
$\ele$ is not in the plus closure of $(\varx^4,\vary^4,\varz^4)$
in $\ring_\vartrans =
\fieldclo(\vartrans)[\varxyz]/(\equa_\vartrans)$, though
it is in the tight closure. So Hochster's tantalizing question has a
negative answer even in dimension two.

This means also that in the theorem that tight closure equals plus closure
for homogeneous primary ideals in a two-dimensional standard-graded
domain over a finite field
\cite[Theorem 4.1]{brennertightplus} we cannot drop the last assumption.
It also implies that in the sequence of vector
bundles
$$\Syz(\varx^{4\po},\vary^{4\po},\varz^{4\po})(6\po)=\Frobitpb
(\Syz(\varx^{4},\vary^{4},\varz^{4})(6)),\, \, \po=2^\expoit \, ,$$
(which are strongly semistable and of degree $0$) on $\Proj \ring_\vartrans$ there are no repetitions of isomorphism types.
\end{remark}

\begin{remark}
Our example has the following implication on the behavior of the
cohomological dimension under equicharacteristic deformations.
For this we consider the open subset
$$\open\!
=\!D(\varx,\vary,\varz)\! \subset\! \Spec \algforc\! \to\! \Spec \ringglobal\! \to\! \AA^1,
\algforc \!
= \!\fieldclo [\vartrans,\varxyz, \varu,\varv,\varw]/
(\equa_\vartrans, \varu\varx^4 + \varv\vary^4 +\varu\varz^4 +\vary^3\varz^3)$$
($\algforc$ is the forcing algebra for the given ideal generators
and the given element; see \cite{hochstersolid} for the definition of forcing algebras and its relation to solid closure).
Then the open subsets $\open_\point$ in the fibers $\Spec
\algforc_\point$, $\point \in \AA^1$, are affine schemes for all closed
points $\point$
(corresponding to algebraic values $\alg \in \fieldclo $),
but this is not the case when the closed point is replaced by the
generic point.
This follows from the cohomological criterion for tight closure and from our example. We do not know
whether such a deformation behavior of cohomological dimension is
possible in characteristic zero.
\end{remark}

\begin{remark}
\label{arithremark}
The geometric deformations should be seen in analogy with arithmetic
deformations of tight closure. For arithmetic deformations the base
space is not $\Spec \fieldclo[\vartrans]$, but $\Spec \ZZ$ (in the
simplest case).
It was shown in
\cite{brennerkatzmanarithmetic} for the ideal $\ideal=(\varx^4,\vary^4,\varz^4)$ and $\ele
=\varx^3 \vary^3$ in $\ringglobal=\ZZ[\varx,\vary,\varz]/(\varx^7+\vary^7-\varz^7)$ that
the containment $\ele \in \ideal^*$ holds in
$\ringglobal_{\ZZ/(\numprim)}=
\ringglobal
\tensor_\ZZ
(\ZZ/(\numprim))$ for $\numprim=3 \modu 7$ and does not hold for $\numprim=2 \modu 7$.
This answered negatively another old question of tight closure
theory and was a guide in constructing our counterexample to the
localization property.
\end{remark}

\begin{remark}
There are probably also similar examples in higher
characteristics. For ${\rm char}(\field)=3$, the second author has
shown in \cite[Theorem 3]{monskyhilbertkunzzd} that the Hilbert-Kunz
multiplicity of $\ring_\alg=\field[\varxyz]/(\equa_\alg)$, where
$\equa_\alg= \varz^4-\varx\vary(\varx+\vary)(\varx+\alg\vary)$ and
$\alg \in \field$, $\alg \neq 0$ or $1$,
is $>3$ or is $3$ according as
$\alg$ is algebraic or transcendental over
$\ZZ/(3)$. However, one can not use the ideals
$(\varx^{3^\expoe}, \vary^{3^\expoe},\varz^{3^\expoe})$ directly, because the
degree bound $ 3 \cdot 3^\expoe/2$ is not an integer. But one can
look for finite ring extensions $\ring_\alg \subseteq \ringsec_\alg$
where there are elements having the critical degree. For example,
look at the ring homomorphism
$$\field[\varxyz,
\vartrans]/(\varz^4-\varx\vary(\varx+\vary)(\varx+\vartrans \vary))
\longto
\field[\varu, \varv, \varw,
\vartrans]/(\varw^8-(\varu^4+\varv^4)(\varu^4+\vartrans \varv^4)) =
\ringglobalsec
$$
given by $\varx \mapsto \varu^4$, $\vary \mapsto \varv^4$ and $\varz \mapsto \varu\varv\varw^2$.
Then the image ideal of $(\varxyz)$ is $\idealsec=( \varu^4,\varv^4, \varu \varv
\varw^2)$, and the stability properties of the syzyzgy bundles
$\Frobitpb (\Syz( \varu^4,\varv^4, \varu \varv\varw^2 ))$
on the projective curves $\Proj \ringsec_\alg$ reflect the stability
properties of
$\Frobitpb (\Syz( \varx,\vary, \varz))$ on
$\Proj \ring_\alg$
(in particular, they depend on whether $\alg$ is algebraic or transcendental).
Therefore every element in $\ringglobalsec$ of degree $6$
(like $\varu^3\varv^3$) lies in $\idealsec^*$ in the generic fiber ring
$\ringsec_\vartrans$, but this might again not be true in the fiber rings over the closed points.
\end{remark}


\begin{remark}
The example has no bearing on the existence of a tight-closure type operation in characteristic zero. For example, solid closure (in characteristic zero) behaves well under geometric
deformations in the graded case in dimension two, since it can be
characterized in terms of the Harder-Narasimhan filtration of the
syzygy bundle (see \cite[Theorem 2.3]{brennertightplus}), and since semistability is an
open property.
\end{remark}

\end{document}